\documentclass{article}
\usepackage{amsthm,amsmath,amssymb}
\usepackage{color}

\RequirePackage[colorlinks,citecolor=blue,urlcolor=blue]{hyperref}

\def\bE{\mathbb E}

\newtheorem{theorem}{Theorem}
\newtheorem{lemma}{Lemma}
\newtheorem{Corollary}{Corollary}

\newtheorem{Assumption}{Assumption}

\begin{document}
%\begin{frontmatter}

\title{Estimation of matrices with row sparsity}
\author{O. Klopp, A. B. Tsybakov}

\maketitle

\begin{abstract}
An increasing number of applications is concerned with recovering a
sparse matrix from noisy observations.   In this paper, we consider the setting
where each row of the unknown matrix is sparse. We establish minimax  optimal rates of convergence for estimating  matrices with row sparsity. A major focus in the present paper is on the derivation of lower bounds.

\end{abstract}

\section{Introduction}
In recent years, there has been a great interest for the theory of estimation in high-dimensional statistical models under different sparsity scenarii. The main motivation behind sparse estimation is based on the observation that, in several practical applications, the number of variables is much larger than the number of observations, but the degree of freedom of the underlying model is relatively small. One  example of such sparse estimation is the problem of estimating of a sparse regression vector from a set of linear measurements (see, e.g., \cite{bickel_ritov_tsybakov}, \cite{bunea_tsybakov_2}, \cite{lounici_lasso}, \cite{van-geer}). Another example is the problem of  matrix recovery under the assumption that the unknown matrix has low rank (see, e.g., \cite{candes-recht-exact, rohde_tsybakov,Koltchinskii-Tsybakov,klopp_general}).

   In some recent papers dealing with covariance matrix estimation, a different notion of sparsity   was considered (see, for example, \cite{cai_zhou}, \cite{rigollet_tsybakov_sinica}). This notion is based on  sparsity assumptions on the rows (or columns) $M_{i\cdot}$  of matrix $M$. One can consider the hard sparsity assumption  meaning that each row $M_{i\cdot}$  of $M$  contains at most $s$ non-zero elements, or soft sparsity assumption, based on imposing a certain decay rate on ordered entries of $M_{i\cdot}$. These notions of sparsity can be defined in terms of $l_q-$balls for $q\in [0,2)$, defined as
  \begin{equation}
  \mathbb B_{q}(s)=\left \{v=(v_i)\in \mathbb R^{n_2}\,:\,\sum^{n_2}_{i=1}\vert v_i\vert^{q}\leq s\right \}
  \end{equation}
where $s<\infty$ is a given constant. The case $q=0$
  \begin{equation}
  \mathbb B_{0}(s)=\left \{v=(v_i)\in \mathbb R^{n_2}\,:\,\sum^{n_2}_{i=1}\mathbb I(v_i\not=0)\leq s\right \}
  \end{equation}
  corresponds to the set of vectors $v$ with at most $s$ non-zero elements. Here $\mathbb{I}(\cdot)$ denotes the indicator function and $s\ge 1$ is an integer.

  In the present note, we consider this row sparsity setting in the matrix signal plus noise model.
Suppose we have noisy observations $Y=(y_{ij})$ of an $n_1\times n_2$ matrix $M=(m_{ij})$ where
 \begin{equation}\label{model}
 %\begin{split}
 y_{ij}=m_{ij}+\xi_{ij}, \quad i=1,\dots,n_1, \ \ j=1,\dots, n_2,
 %\end{split}
\end{equation}
here, $\xi_{ij}$  are i.i.d Gaussian $\mathcal{N}(0,\sigma^{2})$, $\sigma^{2}>0$, or sub-Gaussian random variables.  We denote by $E=(\xi_{ij})$ the corresponding matrix of noise.
We study the  minimax optimal rates of convergence for the estimation of $M$ assuming that  there exist $q\in [0,2)$ and $s$ such that $M_{i\cdot}\in \mathbb B_{q}(s)$ for any $i=1,\dots,n_1$.

 The minimax rate of convergence  characterizes the fundamental limitation of the estimation accuracy. It also captures the interdependence between the different parameters in the model. There is an rich line of work on such fundamental limits (see, for example, \cite{ibragimov_hasm,tsybakov_book,Johnstone}). The minimax risk depends crucially on the choice of the norm in the loss function. In the present paper, we measure the estimation error in  $ \| \cdot\|_{2,p}$-(quasi)norm   for $0<p<\infty$ (for the definition see \eqref{norm_p}).

 For  $n_1=1$, we obtain the problem of estimating of a vector   belonging to a $\mathbb B_{q}(s)$ ball in $\mathbb{R}^{n_2}$. This problem was considered in a number of papers, see, for example, \cite{donoho_johnstone}, \cite{birge_massart}, \cite{abramovich_donoho}, \cite{rigollet_tsybakov_annals}. Let $\eta_{vect}$ denote the minimax rate of convergence with respect to the squared Euclidean norm in the vector case.  It is interesting to note that  the results of the present paper  show that, for the case $p=2$, the minimax rate of convergence for estimation of matrices under the row sparsity assumption is $n_1\eta_{vect}$. Thus, in this case, the problem reduces to estimation of each row separately. The additional matrix structure does not lead to improvement or deterioration of the rate of convergence. We show that it is also true for general $p$.

% More precisely, for an integer $s\in \{1,\dots,n_2\}$ we denote by $\mathcal{A}(s)$  the set of matrices $M\in \mathbb{R}^{n_1\times n_2}$ such that $\vert J(M_{i\cdot})\vert\leq s$ for $i=1,\dots,n_1$. The minimax upper and low bounds in Sections \ref{low_bound} and \ref{upper bounds} imply that
%  \begin{equation}\label{risk}
%    \begin{split}
%    \underset{\hat M}{\inf}\underset{M\in \mathcal{A}(s)}{\sup} \bE\|\hat M-M\|^{2}_{2}\asymp \sigma^2 s n_1\,\log\left (\dfrac{e\,n_2}{s}\right ).
%    \end{split}
%    \end{equation}
%  This is the same rate of convergence as given by \eqref{risk_hard_sparcity}.

 A major focus in  the present paper is on derivation of lower bounds, which is a key step in establishing minimax optimal rates of convergence. Our analysis is based on a new selection lemma (Lemma \ref{selection_lemma}).
 The rest of the paper is organized  as follows. In  Section \ref{notations}, we introduce the notation and some basic tools used throughout the paper. Section \ref{low_bound} establishes the minimax lower bounds for estimation of matrices with row sparsity in  $ \| \cdot\|_{2,p}$-norm, see Theorems  \ref{thm_lower_bound} and \ref{thm_lower_bound_q}.  In Section \ref{minimax}, we derive the upper bounds on the risks using a reduction  to the vector case.
 Most of the proofs are given in the appendix.

% % % % % % % % % % % % % % % % % % % % % % % % % % % % % % % % % % % % % % % % % % % % % %
\subsection{Definitions and notation}\label{notations}
%Here we collect notations, some basics tools and definitions used throughout of the paper.

 Let $A$ be a matrix or a vector. For $0< q< \infty$ and $A\in \mathbb R^{n_1\times n_2}=(a_{ij})$, we denote by $ \| A\|_q=\left (\sum_{i,j}  \vert a_{ij}\vert^q\right )^{1/q}$ the elementwise $l_q$-(quasi-)norm of $A$,  and by $ \left\Vert A\right\Vert_0$  the number of non-zero coefficients of $A$: 
 $$\left\Vert A\right\Vert_0=\sum_{i,j} \mathbb{I}(a_{ij}\not = 0)$$
 where $\mathbb{I}(\cdot)$ denotes the indicator function.
 % and for a matrix $A\in \mathbb R^{n_1\times n_2}$,  $ \| A\|_1$ is the $l_1$-norm of its coefficients.
% We will measure the estimation error in $ \| \cdot\|_{2,p}$-(quasi-)norm.
 For any $A=(A_{1\cdot},\dots, A_{n_1\cdot})^{T}\in \mathbb R^{n_1\times n_2}$ and $p> 0$ define
 \begin{equation}\label{norm_p}
 \begin{split}
  \| A\|_{2,p}=\left (\sum\limits_{i=1}^{n_1}  \| A_{i\cdot}\|^{p}_2\right )^{1/p}.
 \end{split}
 \end{equation}
% and
% $$
% \| A\|_{0,\infty}=\max_{i=1,\dots,n_1}  \| A_{i\cdot}\|_0 = \max_{i=1,\dots,n_1} \sum_{j=1}^{n_2} \mathbb I_{\{a_{ij}\ne 0\}}
% $$
%where $\mathbb I_{\{\cdot\}}$ denotes the indicator function. Thus,
%$$
%\mathcal{A}(s)=\{M \in \mathbb R^{n_1\times n_2} : \,  \| M\|_{0,\infty} \le s \}.
%$$
For $p=2$, $ \| A\|_{2,2}$ is the elementwise $l_2$-norm of $A$ and we will use the notation $\| \cdot\|_{2,2}=\| \cdot\|_{2}$. For $0<p<1$, we have the following inequality
$$\| A+A'\|^{p}_{2,p}\leq \| A\|^{p}_{2,p}+\| A'\|^{p}_{2,p}.$$
%which implies that $\| \cdot\|^{p}_{2,p}$ is a semi-distance.
For $q\in[0,2)$ and $s> 0$ we define the following class of matrices
\begin{equation}\label{class_sparsity}
\mathcal{A}(q,s)=\{A \in \mathbb R^{n_1\times n_2} : \,  A_{i\cdot}\in \mathbb B_q(s)\;\text{for any}\; i=1,\dots,n_1\}.
\end{equation}
In the limiting case $q=0$,  we will also write
\begin{equation}\label{class_sparsity_hard}
\mathcal{A}(s)=\{A \in \mathbb R^{n_1\times n_2} : \,  A_{i\cdot}\in \mathbb B_0(s)\;\text{for any}\; i=1,\dots,n_1\}.
\end{equation}
We set $\mathbb N_{n_1\times n_2}=\left \{(i,j)\,:\,1\leq i\leq n_1,\, 1\leq j\leq n_2\right \}$. For two real numbers $a$ and $b$ we use the notation
$a \wedge b := \min(a, b)$, $a \vee b := \max(a, b)$;
we denote by $\lfloor x\rfloor$  the integer part of $x$; we use the symbol $C$ for a generic positive constant, which is independent of $n_1,n_2,s$  and $\sigma$ and may take different values at different appearances.

% For a matrix $A\in \mathbb{R}^{n_1\times n_2}$ we set $\mathbf{J(A)}=\{(i,j)\in \mathbb N_{n_1\times n_2}\,:\, a_{ij}\not = 0\}$ and for a vector $\omega\in \mathbb{R}^{n_2}$ we set $\mathbf{J(\omega)}=\{1\leq j\leq n_2\,:\, a_{j}\not = 0\}$. %For any $J\subset \mathbb N_{n_1\times n_2}$, $ \| J \|$ denotes its cardinality.
%For any $1\leq i\leq m_1$, let $S_i$ denotes the set of indices of the non-zero elements of the $i-$th row of $A_0$.

 % % % % % % % % % % % % % % % % % % % % % % % % % % % % % % % % % % % % % % % % %

 \section{Lower bounds}\label{low_bound}
We start by establishing the minimax lower bounds for estimation of matrices over the classes ${\cal A}(s)$ (Theorem \ref{thm_lower_bound}) and ${\cal A}(q,s)$ (Theorem \ref{thm_lower_bound_q}). We  denote by $\underset{\hat A}{\inf}$ the infimum over all estimators $\hat A$ with values in $\mathbb{R}^{n_1\times n_2}$. Consider first the case $q=0$. %The following minimax lower bounds holds.
 %Recall that   for an integer $s\in \{1,\dots,n_2\}$ we denote by $\mathcal{A}(s)$  the set of matrices $A=(A_1\,\dots,A_{n_1})^{T}\in \mathbb{R}^{n_1\times n_2}$ such that for any $i\in \{1,\dots,n_{n_1}\}$, $\vert J(A_i)\vert\leq s$.
 \begin{theorem}\label{thm_lower_bound}
Let $n_1,n_2\geq 2$ and $p> 0$. Fix an integer $1\leq s\leq n_2/2$ . Assume that for $(i,j)\in \mathbb N_{n_1\times n_2}$%$i=1,\dots,n_1$,
 %$j=1,\dots,n_2$
  the noise variables $\xi_{ij}$ are i.i.d Gaussian $\mathcal{N}(0,\sigma^{2})$, $\sigma^{2}>0$. Then,
 \begin{itemize}
 \item[(i)]  \begin{equation*}
  \begin{split}
  \underset{\hat A}{\inf}\underset{A\in \mathcal{A}(s)}{\sup}\mathbb P \left \{ \| \hat A-A\|^{2}_{2,p}\geq C\,\sigma^{2}\,(n_1)^{2/p}\,s\,\log\left (\dfrac{e\,n_2}{s}\right )\right \}\geq \beta;
  \end{split}
  \end{equation*}

 \item[(ii)] $$\underset{\hat A}{\inf}\underset{A\in \mathcal{A}(s)}{\sup}\bE \| \hat A-A\|^{2}_{2,p}\geq  \tilde C\,\sigma^{2}\,(n_1)^{2/p}\,s\,\log\left (\dfrac{e\,n_2}{s}\right ).$$
  \end{itemize}
  where $0<\beta<1$, $C>0$, and $\tilde C>0$ are absolute constants.
 \end{theorem}
 \begin{proof}  It is enough to prove (i) since (ii) follows from (i) and Markov inequality.
 For a $A\in \mathbb{R}^{n_1\times n_2}$, we denote by $\mathbb{P}_{A}$ the probability distribution of $\mathcal N(A,\sigma^{2}I)$ Gaussian random vector where $I$ denotes $(n_1n_2)\times (n_1n_2)$ identity matrix. We denote by $\mathrm{KL}(P,Q)$ the Kullback-Leibler divergence between the probability measures $P$ and $Q$.

    To prove (i) we use Theorem 2.5 in \cite{tsybakov_book}. It is enough to check that there exists a finite subset $\Omega'$ of $\mathcal{A}(s)$ such that for any two distinct $B,B'$ in $\Omega'$ we have
    \begin{itemize}

    \item[(a)]$\|  B-B'\|^{2}_{2,p}\geq C\,\sigma^{2}\,(n_1)^{p/2}\,s\,\log\left (\dfrac{e\,n_2}{s}\right )$,
    \item[(b)]$\mathrm{KL}(\mathbb{P}_{B},\mathbb{P}_{B'})\leq \alpha \log\,(\mathrm{card }\,\Omega')$
    \end{itemize}
    for some constants $C>0$ and $0<\alpha<1/8$.

    Denote by
     $\{0,1\}^{s}_{n_1\times n_2}$ the set of all matrices $A=(a_{ij})\in \mathbb{R}^{n_1\times n_2}$ such that $a_{ij}\in \{0,1\}$ and each row of $A$ contains exactly $s$ ones. For any two matrices $A=(a_{ij})$ and $A'=(a'_{ij})$ in $\{0,1\}^{s}_{n_1\times n_2}$ define the Hamming distance $$\mathrm{d}_{H}(A,A')=\sum\limits_{(i,j)\in \mathbb N_{n_1\times n_2}}\mathbb I_{\{a_{ij}\not =a'_{ij}\}}.
     $$
     %where $\mathbb I_{\{a_{ij}\not =a'_{ij}\}}=1$ if $a_{ij}\not =a'_{ij}$ and $\mathbb I_{\{a_{ij}\not =a'_{ij}\}}=0$ otherwise.
    We use of  the following selection lemma proved in Appendix \ref{proof_selection_lemma}.
    \begin{lemma}\label{selection_lemma}
    Let $n_1,n_2\geq 2$ and $1\leq s\leq n_2/2$. Then, there exists a subset $\Omega$ of  $\{0,1\}^{s}_{n_1\times n_2}$ such that for some numerical constant $C\geq 10^{-5}$
    \begin{equation}\label{lower_1}
           \log (|\Omega|)\geq C\,n_1\,s\,\log\left (\dfrac{e\,n_2}{s}\right )
           \end{equation}
     and, for any two distinct $A,A'$ in $\Omega$, the Hamming distance satisfies
      \begin{equation}\label{distance}
      \begin{split}
      \mathrm{d}_{H}(A,A')\geq \dfrac{n_1\,(s+1)}{16}.
      \end{split}
      \end{equation}
    \end{lemma}
     Fix $0<\gamma<1$ and define
    $$\Omega'=\left \{\sigma\,\gamma\,\sqrt{\log\left (\dfrac{e\,n_2}{s}\right )}\,A\quad:\quad A\in \Omega\right \}$$
    where $\Omega$ is a set satisfying the conditions of Lemma \ref{selection_lemma}.
 For $p=2$ using \eqref{distance} we obtain that for any two distinct $B,B'$ in $\Omega'$
    $$
    \| B-B'\|^{2}_{2}\geq \dfrac{\gamma^{2}\,\sigma^{2}\,n_1\,s\,}{16}\log\left (\dfrac{e\,n_2}{s}\right ).
    $$
    This implies (a) for $p=2$. For $p\not = 2$ we will use the following elementary lemma, cf.  Appendix \ref{proof_lemma_lower_bound_p}.
    \begin{lemma} \label{lemma_lower_bound_p}
    If $A=(a_{ij})$ and $A'=(a'_{ij})$ are two elements of $\{0,1\}^{s}_{n_1\times n_2}$ such that $\mathrm{d}_{H}(A,A')\geq \dfrac{n_1\,(s+1)}{16}$, then the cardinality of the set $J(A,A')=\left \{1\leq i\leq n_1\,:\,\sum\limits_{j=1}^{n_2}\mathbb{I}_{\{a_{ij}\not=a'_{ij}\}}>\dfrac{s}{32}\right \}$ is greater than or equal to $\dfrac{n_1}{64}$.
    \end{lemma}
    Lemma \ref{lemma_lower_bound_p} implies that for any two distinct $B,B'$ in $\Omega'$
     \begin{equation}\label{norm_2,p}
     \begin{split}
     \| B-B'\|^{2}_{2,p}&\geq \gamma^{2}\,\sigma^{2}\,\log\left (\dfrac{e\,n_2}{s}\right )\left (\left (\dfrac{s}{32}\right )^{p/2}\dfrac{n_1}{64}\right )^{2/p}\\&\geq \dfrac{\gamma^{2}\,\sigma^{2}}{64^{1+2/p}}n_1^{2/p}s\log\left (\dfrac{e\,n_2}{s}\right ),
     \end{split}
     \end{equation}
which yields (a) for $p\not=2$.

    To check (b), note that $\mathrm{d}_{H}(A,A')\le 2n_1s$ for all $A,A'\in \{0,1\}^{s}_{n_1\times n_2}$. This implies
    \begin{equation}\label{lower_2}
    \mathrm{KL}(\mathbb{P}_{B},\mathbb{P}_{B'})=\dfrac{1}{2\,\sigma^{2}}\| B-B'\|^{2}_{2}\leq \gamma^{2}\,n_1\,s\,\log\left (\dfrac{e\,n_2}{s}\right ).
    \end{equation}
Since also $|\Omega|=|\Omega'|$,
    from \eqref{lower_1} and \eqref{lower_2} we deduce that (b) is satisfied with $\alpha<1/8$ if $\gamma>0$ is chosen sufficiently small. This completes the proof of Theorem \ref{thm_lower_bound}.
 \end{proof}

Note that there are $\binom{n_2}{s}^{n_1}$ possible sparsity patterns which satisfy the hard sparsity condition on the rows. By standard bounds on binomial coefficients, we have  $\log\left (\binom{n_2}{s}^{n_1}\right )\asymp n_1s\log\left (\frac{n_2}{s}\right )$. Consequently, the rate $n_1s\log\left (\frac{en_2}{s}\right )$ corresponds to the logarithm of the number of models.

Let us turn out to the soft sparsity scenario. For any $0<q< 2$ and $s>0$ define the quantity
 \begin{equation}\label{lower_bound_q}
 \begin{split}
 \eta(s)=\left (n_1\,s\,\left [\sigma^{2}\,\log\left (1+\frac{\sigma^{q}\,n_2}{s}\right )\right ]^{1-q/2}\right )\vee\left ( n_1\,s^{2/q}\right )\vee \left (n_1\,n_2\,\sigma^{2}\right )
 \end{split}
 \end{equation}

 The minimax  lower bound is given by the following theorem proved  in Appendix \ref{proof_lower_bound_q}.
%Let $N\in\mathbb{N}$ be such that
%$$\dfrac{N}{\log(e\,N)}\geq 2\sigma^{2}.$$(s)
 \begin{theorem}\label{thm_lower_bound_q}
 Let $n_1, n_2\geq 2$. Fix $0< q< 2$ and $s>0$. Suppose that for $(i,j)\in \mathbb N_{n_1\times n_2}$
      the noise variables $\xi_{ij}$ are i.i.d Gaussian $\mathcal{N}(0,\sigma^{2})$, $\sigma^{2}>0$.
 Then, there exists a numerical constant $c^{*}$ such that
  \begin{itemize}
  \item[(i)] \begin{equation*}
   \begin{split}
   \underset{\hat A}{\inf}\underset{A\in \mathcal{A}(q,s)}{\sup}\mathbb P \left \{ \| \hat A-A\|^{2}_{2}\geq c^{*}\,\eta(s)\right \}\geq \beta,
   \end{split}
   \end{equation*}
   where $0<\beta<1$ and
    \item[(ii)]
  $$\underset{\hat A}{\inf}\underset{A\in \mathcal{A}(q,\delta)}{\sup}\bE \| \hat A-A\|^{2}_{2}\geq c^{*}\,\eta(s).$$
   \end{itemize}
  \end{theorem}

   \section{Minimax rates of convergence}\label{minimax}
   Consider the problem of estimating of a vector $v=(v_i)\in \mathbb B_{q}(s)\subset \mathbb{R}^{n_2}$ from noisy observations
   \begin{equation*} %\label{model}
    %\begin{split}
    y_{i}=v_{i}+\xi_{i}, \quad i=1,\dots,n_2,
    %\end{split}
   \end{equation*}
   where $\xi_{ij}$ are i.i.d. Gaussian $\mathcal{N}(0,\sigma^{2})$, $\sigma^{2}>0$.

    The non-asymptotic minimax optimal rate of convergence for estimation of $v$ in the $l_2-$norm, obtained in \cite{birge_massart}, is given by
    \begin{equation*}
    \eta_{vect}(s)=\sigma^{2}\,s\,\log\left (\dfrac{e\,n_2}{s}\right )
    \end{equation*}
    when $q=0$ and by
    \begin{equation*} %\label{lower_bound_q}
     \begin{split}
     \eta_{vect}(s)=\left (s\,\left [\sigma^{2}\,\log\left (1+\frac{\sigma^{q}\,n_2}{s}\right )\right ]^{1-q/2}\right )\vee\left ( s^{2/q}\right )\vee \left (n_2\,\sigma^{2}\right )
     \end{split}
     \end{equation*}
     when $0<q<2$.

      We see that, for $p=2$, the lower bounds given by Theorems \ref{thm_lower_bound} and \ref{thm_lower_bound_q} are $n_1 \eta_{vect}(s)$ in the case of hard sparsity and $n_1\eta_{vect}(s)$ in the case of soft sparsity. We get the same rate as when estimating  each row separately. This implies that, in this particular case, the additional matrix structure  does not lead to improvement or to deterioration of the rate of convergence.

As shown below and in view of the lower bounds of Theorems \ref{thm_lower_bound} and \ref{thm_lower_bound_q}, optimal rates for arbitrary $p$ can be also obtained from vector estimation method. It suffices to  apply to the rows of $M$ a minimax optimal method for vector estimation on $\mathbb B_{q}(s)$ balls.
One can take, for example, the  following penalized least squares estimator $\hat M$ of $M$ (cf. \cite{birge_massart}):
\begin{equation}\label{estim_hard}
\hat M=\underset{A\in \mathbb R^{n_1\times n_2}}{\mathrm{argmin}}\left \{ \left\Vert Y-A\right\Vert^{2}_{2}+\lambda\Vert A\Vert_0\log\left (\dfrac{e\,n_1\,n_2}{\Vert A\Vert_0\vee 1}\right )\right \}
\end{equation}
where $\lambda>0 $ is a regularization parameter. The penalty in \eqref{estim_hard} is inspired by the hard thresholding penalty $\Vert A\Vert_0$, which leads to $\hat m_{ij}$ that are thresholded values of $y_{ij}$ (see, for instance \cite{hardle}, page 138).

The penalized least squares estimator defined in \eqref{estim_hard} can be computed efficiently. Let $y_{(j)}$  denote the $j$th largest in absolute value component of $Y$. The estimator $\hat M$ is obtained by thresholding the coefficients of $Y$: we keep  $y_{(j)}$ such that
\begin{equation*}
y^{2}_{(j)}>\lambda\left ( \log(e\,n_1\,n_2)+\sum^{j}_{i=2}(-1)^{i+j+1}\,i\,\log(i)\right )
\end{equation*}
and set all other coefficients equal to zero.

In what follows we assume that the noise variables $\xi_{ij}$  are zero-mean and sub-Gaussian, which means that they satisfy the following assumption.
   \begin{Assumption}\label{ass_noise} $\bE(\xi_{ij})=0$ and there exists a constant $K>0$ such that
   $$
   \left (\bE \vert \xi_{ij}\vert^{p}\right )^{1/p}\leq  K\sqrt{p}\quad \text{for all}\quad p\geq 1$$
   for any $1\leq i\leq n_1$ and $1\leq j\leq n_2$.
   \end{Assumption}
 This assumption on the noise variables means that their distribution  is dominated by the distribution of a centered Gaussian random variable.
  This class of distributions is rather wide. Examples of sub-Gaussian random variables are Gaussian or
  bounded random variables. In particular, Assumption \ref{ass_noise} implies that $\bE \left (\xi_{ij}^{2}\right )\leq2\, K^{2}$.

The next theorem presents oracle inequalities for the penalized least squares estimator $\hat M$, both in probability and in expectation.
\begin{theorem}\label{thm_upper_bounds}
Let $\hat M$ be the penalized least squares estimator defined in \eqref{estim_hard}, $a>1$ and  $\lambda=2a\,K_0\,K^{2}$ where $K_0>0$ is large enough. Suppose that Assumption \ref{ass_noise} holds. Then,  for any $\Delta>0$
\begin{equation}\label{upper_proba}
 \Vert M-\hat M\Vert^{2}_{2}\leq \underset{A\in \mathbb R^{n_1\times n_2}}{\inf}\left \{ \frac{a+1}{a-1}\Vert M-A\Vert^{2}_{2}+C\,K^{2}\,\Vert A\Vert_{0}\,\log\left (\dfrac{e\,n_1\,n_2}{\Vert A\Vert_{0}\vee 1}\right) \right \}+\dfrac{2\,a^{2}}{a-1}\Delta
\end{equation}
with probability at least $1-2\exp\left \{-\frac{C_0\,\Delta}{K^{2}}\right \}$, and
\begin{equation}\label{upper_expactation}
\mathbb E \,\Vert M-\hat M\Vert^{2}_{2}\leq \underset{A\in \mathbb R^{n_1\times n_2}}{\inf}\left \{ \frac{a+1}{a-1}\Vert M-A\Vert^{2}_{2}+C\,K^{2}\Vert A\Vert_{0}\,\log\left (\dfrac{e\,n_1\,n_2}{\Vert A\Vert_{0}\vee 1}\right) \right \}+\tilde C\,K^{2}
 \end{equation}
 where $C,C_0$ and $\tilde C$ are numerical constants.
\end{theorem}
 For the particular case of Gaussian noise, the result \eqref{upper_expactation} of Theorem \ref{thm_upper_bounds} is proved in \cite{birge_massart}, and the result \eqref{upper_proba} in \cite{bunea_tsybakov_0}.   Theorem \ref{thm_upper_bounds} extends the analysis to the case of sub-Gaussian noise. The prooof is given in Appendix~\ref{proof_thm_upper_bounds}.

 Now suppose that $M\in \mathcal{A}(s)$. Using Theorem \ref{thm_upper_bounds} and the inequality $$ \| \hat M-M\|_{2,p}\leq n^{1/p-1/2}_1\| \hat M-M\|_{2}$$ that holds for any $0< p\le 2$ we obtain the following corollary.
   \begin{Corollary}\label{corollary_upper_bound}
   Let $\hat M$ be the penalized least squares estimator defined in \eqref{estim_hard} with $\lambda=K_0\,K^{2}$ where $K_0>0$ is large enough. Suppose that Assumption \ref{ass_noise} holds and that $M\in \mathcal{A}(s)$. Then,  for all $0< p\leq 2$ and for any $\Delta>0$
   \begin{equation} \label{thm_upper_3a}
                                          \begin{split}
                                          \| \hat M-M\|_{2,p}^2\leq C\,K^{2}\,n^{2/p}_1 s\log\left (\dfrac{e\,n_2}{s}\right )+\Delta
                                          \end{split}                                      \end{equation}
                                         with probability at least $1-2\exp\left \{-\frac{C_2\,\Delta}{K^{2}}\right \}$, and

    \begin{equation}\label{thm_upper_4a}
                                          \begin{split}
                                         \bE \| \hat M-M\|^{2}_{2,p}\leq
                                           C\,K^{2}\,n^{2/p}_1 s\log\left (\dfrac{e\,n_2}{s}\right ).
                                          \end{split}                                      \end{equation}
   \end{Corollary}
%   Comparing  Corollary  \ref{corollary_upper_bound} with the lower bounds of Theorem~\ref{thm_lower_bound}, we can conclude that, in the case of Gaussian errors,
% % , on the class of matrices $\mathcal{A}(s)$ and
%  the rate of convergence $n^{2/p}_1\,s\log\left (\dfrac{e\,n_2}{s}\right )$  is minimax optimal  for $0< p\leq 2$, and the aggregate estimator $\tilde M$ attains this rate.
%
  These inequalities shows that, for $0<p\leq 2$, the penalized least squares estimator \eqref{estim_hard} achieves the rate of convergence given by Theorem \ref{thm_lower_bound}.This implies that this rate   is minimax optimal.

 The next corollary shows that the estimator \eqref{estim_hard} also achieves the minimax rate of convergence
 in a more general setting when  $M\in \mathcal{A}(q,s)$ for  $0<q< 2$.   For any $0<q< 2$ and $s>0$ define the quantity
   \begin{equation}\label{upper_bound_q}
   \begin{split}
   \psi(s)=\left (n_1\,s\,\left [K^{2}\,\log\left (1+\frac{K^{q}\,n_2}{s}\right )\right ]^{1-q/2}\right )\vee\left ( n_1\,s^{2/q}\right )\vee \left (n_1\,n_2\,K^{2}\right ).
   \end{split}
   \end{equation}

 \begin{Corollary}\label{corollary_upper_bounds_q}
  Let $\hat M$ be the penalized least squares estimator defined in \eqref{estim_hard} with $\lambda=K_0\,K^{2}$ where $K_0>0$ is large enough.  Suppose that Assumption \ref{ass_noise} holds and $M\in \mathcal{A}(q,s)$. Then, there exists numerical constant $C^{*}$ such that for any $\Delta>0$
    \begin{equation*} % \label{thm_upper_3a}
                                           \begin{split}
                                           \| \hat M-M\|_{2}^2\leq C^{*}\,\psi(s) +\Delta
                                                                                      \end{split}                                      \end{equation*}
                                                                                      with probability at least $1-2\exp\left \{-\frac{C_2\,\Delta}{K^{2}}\right \}$, and
                                               \begin{equation*} %\label{thm_upper_4a}
                                           \begin{split}
                                          \bE \| \tilde M-M\|^{2}_{2,p}\leq
                                           C^{*}\,\psi(s).
                                           \end{split}                                      \end{equation*}

 \end{Corollary}
  We give the proof of Corollary \ref{corollary_upper_bounds_q} in Appendix \ref{proof_corollary_upper_bounds_q}.   If the noise variables $\xi_{ij}$ are i.i.d Gaussian $\mathcal{N}(0,\sigma^{2})$, we have $\psi(s)=\eta(s)$. Thus, the rate of convergence given by \eqref{lower_bound_q}  is minimax optimal.

%Corollary \ref{corollary_upper_bounds_q} and Theorem \ref{thm_lower_bound_q} imply that the rate of convergence given by \eqref{lower_bound_q}  is minimax optimal.

   % % % % % % % % % % % % % % % % % % % % % % % % % % % % % % % % % % % % % % % % % % % % % % % %

\appendix
   % % % % % % % % % % % % % % % % % % % % % % % % % % % % % % % %
   \section{Proof of Lemma \ref{selection_lemma}}\label{proof_selection_lemma}
To prove Lemma \ref{selection_lemma} we use the Varshamov-Gilbert bound. The volume (cardinality) $V_1$  of $\{0,1\}^{s}_{n_1\times n_2}$ is
$$V_1=\binom{n_2}{s}^{n_1}.$$
Note that the volume of the Hamming ball of radius $n_1(s+1)/2$ in $\{0,1\}^{s}_{n_1\times n_2}$ is smaller than the volume $V_2$ of the Hamming ball of the same radius in a larger space  of all matrices $A=(a_{ij})\in \mathbb{R}^{n_1\times n_2}$ such that $a_{ij}\in\{0,1\}$ and $A$ contains at most $n_1s$ ones. Let $K=\left \lfloor \dfrac{n_1(s+1)}{2} \right \rfloor $ where  $\lfloor x \rfloor$ denotes the integer part of $x$.
% W. l. g. we can suppose that $K\geq 1$ (when $K=0$, $\log\vert \Omega\vert =\log V_1\geq cn_1s \log\left  (\dfrac{en_2}{s}\right )$ follows directly from Stirling's formula).
A standard bound implies
$$V_2=\sum^{K}_{i=1}\binom{n_1n_2}{i}\leq \left (\dfrac{en_1n_2}{K}\right )^{K}\leq \left (\dfrac{2en_2 }{s+1}\right )^{n_1(s+1)/2}$$
where we use that $f(x)=x\log\left (\dfrac{en_1n_2}{x}\right )$ is growing for $x\leq n_1n_2$.

In order to lower bound $V_1$ we use
 Stirling's formula (see, e.g., \cite[p. 54]{feller}): for any $j\in \mathbb{N}$
 \begin{equation}\label{stirling}
 \begin{split}
 j!=j^{j+1/2}e^{-j}\sqrt{2\pi}\,\psi(j)\quad \text{with}\\
 e^{(12\,j+1)^{-1}}<\psi(j)<e^{(12\,j)^{-1}}.
 \end{split}
 \end{equation}
Using \eqref{stirling} we get
\begin{equation}\label{selection_4}
\binom{n_2}{s}\geq \dfrac{e^{-1/6}\,\left (\dfrac{n_2}{s}\right )^{n_2+1/2}}{\sqrt{2\pi\,s}\left (\dfrac{n_2}{s}-1\right )^{n_2-s+1/2}}.
\end{equation}
Now, the Varshamov-Gilbert bound implies that there exists a subset $\Omega$ of $\{0,1\}^{s}_{n_1\times n_2}$ such that  $\mathrm{d}_{H}(A,A')> \frac{n_1(s+1)}{2}$ for any $A,A'\in \Omega$, $A\not =A'$ and

\begin{equation*}
\begin{split}
\vert \Omega\vert\geq \dfrac{\binom{n_2}{s}^{n_1}}{\left (\dfrac{2en_2}{s+1}\right )^{n_1(s+1)/2}}\geq \left (\dfrac{e^{-1/6}\,\left (\dfrac{n_2}{s}\right )^{n_2+1/2}(s+1)^{\frac{s+1}{2}}}{\sqrt{2\pi\,s}\left (\dfrac{n_2}{s}-1\right )^{n_2-s+1/2}(2en_2 )^{\frac{s+1}{2}}}\right )^{n_1}
\end{split}
\end{equation*}
which implies
\begin{equation}\label{VGB}
\begin{split}
\log \vert \Omega\vert&\geq n_1\left [-\dfrac{1}{6}-\frac{1}{2}\log s-\log(\sqrt{2\pi})+(n_2+1/2)\log \left (\dfrac{n_2}{s}\right )+\frac{s+1}{2}\log (s+1)\right .\\&\left .-(n_2-s+1/2)\log\left (\dfrac{n_2}{s}-1\right )-\frac{s+1}{2}\log (2en_2)\right ] \\&\geq
n_1\left [-\dfrac{1}{6}-\frac{1}{2}\log s-\log(\sqrt{2\pi})+s\log\left (\dfrac{n_2}{s}-1\right )-\frac{s+1}{2}\log \left (\dfrac{2en_2}{s+1}\right )\right ].
\end{split}
\end{equation}
1) We first consider the case $501\leq s\leq n_2/8$.
Using that $\frac{251s}{501}\geq \frac{s+1}{2}$ for $s\geq 501$,   we get
$$\dfrac{s+1}{2}\log \left (\dfrac{2en_2}{s+1}\right )\leq \frac{251s}{501}\log \left (\dfrac{501en_2}{251s}\right )\leq \dfrac{98s}{100}\log \left (\dfrac{n_2}{s}-1\right )$$
where the last inequality is valid for $n_2/s\geq 8$.

On the other hand, it is easy to see that for  $501\leq s\leq n_2/4$ we have $$\frac{1}{2}\log s\leq 0,007 s\log\left (\dfrac{n_2}{s}-1\right )\quad \text{and} \quad \dfrac{1}{6}+\log(\sqrt{2\pi})\leq 0,002 s\log\left (\dfrac{n_2}{s}-1\right ).$$
Then, \eqref{VGB} implies
\begin{equation*}
\begin{split}
\log \vert \Omega\vert&\geq 0.011 n_1s\log\left (\dfrac{n_2}{s}-1\right )\geq 0.01 n_1s\log\left (\dfrac{en_2}{s}\right ).
\end{split}
\end{equation*}
for $n_2/8\geq s\geq 501$.

2) Consider next the case $s<501$ and $s\leq n_2/8$. Now, instead of the set $\{0,1\}^{s}_{n_1\times n_2}$ we will deal with the set
$\{0,1\}^{1}_{n_1\times l}$ where $l=\lfloor n_2/s \rfloor$.  Using the same arguments as above, we will show that there exists a subset $\tilde \Omega \subset \{0,1\}^{1}_{n_1\times l}$ such that $\mathrm{d}_{H}(A,A')\geq n_1/2$ for any $A,A'\in \tilde \Omega $, $A\not =A'$ and $\log (\mathrm{card}\,\tilde\Omega)\geq C\,n_1\,\log\left (e\,n_2\right )$.  In this case, the previous values $V_1$ and $V_2$ are replaced by
$$V_1=l^{n_1},\quad \quad V_2=\sum^{\lfloor n_1/2 \rfloor}_{i=1}\binom{n_1l}{i}\leq  \left (2el\right )^{n_1/2}$$
and
\begin{equation*}
\begin{split}
\log \vert \tilde\Omega\vert&\geq \dfrac{n_1}{2}\left (2\log\left (l\right )-\log\left (2el\right )\right )\geq \dfrac{n_1\log(l)}{10}\geq  10^{-4}n_1s\log\left (\dfrac{en_2}{s}\right )
\end{split}
\end{equation*}
for $s<501$ and $n_2/s\geq 8$.
To embed $\tilde\Omega$ in $\{0,1\}^{s}_{n_1\times n_2}$ define
 $$\Omega=\{A\in \{0,1\}^{s}_{n_1\times n_2}\,:\,A=(\underset{s\,\text{times}}{\underbrace{\tilde A,\dots,\tilde A}},\mathbf{0})\;,\;\tilde A\in \tilde \Omega \;,\;\mathbf{0}\in\mathbb{R}^{n_1\times(n_2-ls)}\}. $$
 We have $\Omega\subset \{0,1\}^{s}_{n_1\times n_2}$, $\mathrm{card}\,\Omega=\mathrm{card}\,\tilde \Omega$  and $\mathrm{d}_{H}(A,A')\geq \dfrac{n_1(s+1)}{4}$ for any $A,A'\in \Omega$, $A\not =A'$.

3) In order to deal with the case $n_2/8\leq s\leq n_2/4.5$ define $s'=\left \lfloor \dfrac{s}{2}\right \rfloor$ and $n'_2=n_2-(s-s')$. Then, $n'_2\geq 8s'$ and we can apply the previous result. This implies that there exists a subset $\bar\Omega$ of $\{0,1\}^{s'}_{n_1\times n'_2}$  such that $$\mathrm{d}_{H}(A,A')\geq\dfrac{n_1(s'+1)}{2} \geq \dfrac{n_1(s+1)}{4}$$ for any $A,A'\in \bar\Omega$, $A\not =A'$ and
  $$\log (\mathrm{card}\,\bar\Omega)\geq 10^{-4} n_1\,s'\,\log\left (\dfrac{e\,n'_2}{s'}\right )\geq \frac{10^{-4}}{2} n_1\,s\,\log\left (\dfrac{e\,n_2}{s}\right )$$
                                                       where we used $n'_2/s'\geq n_2/s$.

                                                       To embed $\bar\Omega$ in $\{0,1\}^{s}_{n_1\times n_2}$ define
 $$\Omega=\{A\in \{0,1\}^{s}_{n_1\times n_2}\,:\,A=(\bar A,\underset{s-s'\,\text{times}}{\underbrace{\mathbf{1},\dots,\mathbf{1}}})\;,\;\bar A\in \bar \Omega \;,\;\mathbf{1}=(1,\dots,1)^{T}\in\mathbb{R}^{n_1}\}. $$

 We have $\Omega\subset \{0,1\}^{s}_{n_1\times n_2}$, $\mathrm{card}\,\Omega=\mathrm{card}\,\bar \Omega$  and $\mathrm{d}_{H}(A,A')\geq \dfrac{n_1(s+1)}{4}$ for any $A,A'\in \Omega$, $A\not =A'$.

 Using exactly the same argument we can treat cases $n_2/4.5\leq s\leq n_2/3$ and $n_2/3\leq s\leq n_2/2$
 to  get the statement of Lemma \ref{selection_lemma}.
% % % % % % % % % % % % % % % % % % % % % % % % % % % % % % % % % % % % % % % % % % % % % % % % %
% % % % % % % % % % % % % % % % % % % % % % % % % % % % % % % % % % % % % % % % % % % % % % % % % % % %
\section{Proof of Lemma \ref{lemma_lower_bound_p}}\label{proof_lemma_lower_bound_p}
Assume that $\mathrm{card}\left (J(A,A')\right )<\dfrac{n_1}{64}$. Then, denoting by $J^{C}(A,A')$ the complement of $J(A,A')$ and using that $\mathrm{card}\left (J^{C}(A,A')\right )\leq n_1$, we get
  \begin{equation*}
  \begin{split}
  \mathrm{d}_{H}(A,A')&\leq 2s\,\mathrm{card}\left (J(A,A')\right )+\dfrac{s}{32}\mathrm{card}\left (J^{C}(A,A')\right )\\
  &<2s\,\dfrac{n_1}{64}+\dfrac{n_1s}{32}=\dfrac{n_1s}{16}
  \end{split}
  \end{equation*}
  which contradicts the premise of the lemma.
% % % % % % % % % % % % % % % % % % % % % % % % % % % % % % % % % % % % % % % % % % % % %
% % % % % % % % % % % % % % % % % % % % % % % % % % % % % % % % % %
% % % % % % % % % % % % % % % % % % % % % % % % % % % % % % % % % % % % % % % % % % % % %
% % % % % % % % % % % % % % % % % % % % % % % % % % % % % % % % % % % % % %
   \section{Proof of Theorem \ref{thm_lower_bound_q}.}  \label{proof_lower_bound_q}
It is enough to prove (i) since (ii) follows from (i) and the Markov inequality.

To prove (i) we use Theorem 2.5 in \cite{tsybakov_book}.  We define $k\geq 1$ be the largest integer satisfying
  \begin{equation}\label{lowes_1}
  \begin{split}
   k\leq s\,\sigma^{-q}\,\left (\log\left (1+\dfrac{n_2}{k}\right )\right )^{-q/2}.
  \end{split}
  \end{equation}
  If there is no $k\geq 1$ satisfying \eqref{lowes_1}, take $k=0$. Set $\bar k=k\vee 1$ and $S=\bar k\wedge \frac{n_2}{2}$.
  Let $\Omega'\subset \{0,1\}^{S}_{n_1\times n_2}$ be the set given by Lemma \ref{selection_lemma}. We consider
  \begin{equation*}
  \Omega=\left \{\tau\left (\frac{\bar\delta}{S}\right )^{1/q}A\,:\,A\in \Omega'\right \}
  \end{equation*}
  where $0<\tau<1$ and $0<\bar \delta\leq s$ will be chosen later. It is easy to see that $\Omega\subset\mathcal{A}(q,s)$.

  Since
    the noise variables $\xi_{ij}$ are i.i.d Gaussian $\mathcal{N}(0,\sigma^{2})$, for any two distinct $B,B'$ in $\Omega$, the Kullback-Leibler divergence $\mathrm{KL}(\mathbb P_{B},\mathbb P_{B'})$ between $\mathbb P_{B}$
    and $\mathbb P_{B'}$ is given by
    \begin{equation}
    \mathrm{KL}(\mathbb P_{B},\mathbb P_{B'})=\dfrac{\Vert B-B'\Vert^{2}_{2}}{2\,\sigma^{2}}
    \end{equation}

  We consider now three cases, depending on the value of the integer $k$ defined in \eqref{lowes_1}.

  \textit{Case (1)}: $k=0$. Since $k=0$,  the inequality \eqref{lowes_1} is violated for $k=1$, so that
  \begin{equation}\label{lowes_3}
  s\leq \sigma^{q}\,\left (\log\left (1+n_2\right )\right )^{q/2}.
  \end{equation}
  Here $S=1$ and we take $\bar \delta=s$.
     We have that
    for any two distinct $B,B'$ in $\Omega$,
        \begin{equation}\label{lowes_2}
        \begin{split}
         \| B-B'\|^{2}_{2} &\geq\dfrac{n_1\tau^{2}}{4.5}\,\left (s\right )^{2/q}.
        \end{split}
        \end{equation}

 On the other hand, by Lemma \ref{selection_lemma}, we have that
 \begin{equation*}
 \log \left \vert \Omega\right \vert\geq C\,n_1\log\left (1+n_2\right )
 \end{equation*}
 and using \eqref{lowes_3}
 \begin{equation}\label{lowes_6}
 \begin{split}
     \mathrm{KL}(\mathbb{P}_{B},\mathbb{P}_{B'})&=\dfrac{1}{2\,\sigma^{2}}\| B-B'\|^{2}_{2}\leq \dfrac{\tau^{2}\,n_1\,s^{2/q}}{\sigma^{2}}\\& \leq \tau^{2}\,n_1\,\log(1+n_2)\\&\leq \alpha \log\left \vert\Omega\right \vert
     \end{split}
     \end{equation}
              for some $0<\alpha<1/8$ if $0<\tau<1$ is chosen sufficiently small.

              \textit{Case (2)}: $1\leq k\leq n_2/2$. We take $\bar \delta=\left (\frac{s}{S}\right )^{1/q}$. For any two distinct $B,B'$ in $\Omega$,
                      \begin{equation}\label{lowes_4}
                      \begin{split}
                  \| B-B'\|^{2}_{2} &\geq\dfrac{n_1\tau^{2}\,(S+1)}{9}\,\left (\frac{s}{S}\right )^{2/q}\\&
                       \geq \dfrac{n_1\tau^{2}}{9}\,(s)^{2/q}\left (s\,\sigma^{-q}\,\left (\log\left (1+\dfrac{n_2}{k}\right )\right )^{-q/2}\right )^{1-2/q}\\&
                                              \geq \dfrac{n_1\tau^{2}}{9}\,s\,\sigma^{2-q}\,\left (\log\left (1+\dfrac{n_2}{k}\right )\right )^{1-q/2}\\&
                                              \geq \dfrac{n_1\tau^{2}}{9}\,s\,\sigma^{2-q}\,\left (\log\left (1+n_2\,s^{-1}\,\sigma^{q}\right )\right )^{1-q/2}.
                      \end{split}
                      \end{equation}

               By Lemma \ref{selection_lemma}, we have that
                \begin{equation*}
                \begin{split}
                \log \left \vert \Omega\right \vert&\geq C\,n_1\,S\,\log\left (1+\dfrac{n_2}{S}\right )\\&\geq
                               \frac{C\,n_1}{2}\,s\,\sigma^{-q}\left (\log\left (1+n_2\,s^{-1}\,\sigma^{q}\right )\right )^{1-q/2}
                \end{split}
                \end{equation*}
               and
               \begin{equation}\label{lowes_5}
               \begin{split}
                   \mathrm{KL}(\mathbb{P}_{B},\mathbb{P}_{B'})&=\dfrac{1}{2\,\sigma^{2}}\| B-B'\|^{2}_{2}\leq \dfrac{\tau^{2}\,n_1}{\sigma^{2}}s^{2/q}\,S^{1-2/q}\\& \leq \dfrac{\tau^{2}\,n_1}{\sigma^{2}}s^{2/q}\,\left (s\,\sigma^{-q}\,\left (\log\left (1+n_2\,s^{-1}\,\sigma^{q}\right )\right )^{-q/2}\right )^{1-2/q}
                   \\& \leq \tau^{2}\,n_1\,\sigma^{-q}\,\left (\log\left (1+n_2\,s^{-1}\,\sigma^{q}\right )\right )^{1-q/2}
                   \\&\leq \alpha \log\left \vert\Omega\right \vert
                   \end{split}
                   \end{equation}
                            for some $0<\alpha<1/8$ if $0<\tau<1$ is chosen sufficiently small.

 \textit{Case (3)}: $ k> n_2/2$.  Since $k> n_2/2$,  the inequality \eqref{lowes_1} is violated for $k= n_2/2$, so that
   \begin{equation}\label{lowes_7}
    s\geq \frac{n_2\,\sigma^{q}}{2}.
   \end{equation}
   In this case $S=n_2/2$ and, using \eqref{lowes_7}, we can take $\bar \delta=\frac{n_2\,\sigma^{q}}{2}$.
      We have that
     for any two distinct $B,B'$ in $\Omega$,
         \begin{equation}\label{lowes_8}
         \begin{split}
          \| B-B'\|^{2}_{2} &\geq \frac{\tau^{2}n_1\,n_2\,\sigma^{2}}{18}.
         \end{split}
         \end{equation}

  On the other hand, by Lemma \ref{selection_lemma}, we have that
  \begin{equation*}
  \log \left \vert \Omega\right \vert\geq C\,n_1\,n_2
  \end{equation*}
  and
  \begin{equation}\label{lowes_9}
  \begin{split}
      \mathrm{KL}(\mathbb{P}_{B},\mathbb{P}_{B'})&=\dfrac{1}{2\,\sigma^{2}}\| B-B'\|^{2}_{2}\leq \dfrac{\tau^{2}\,n_1\,n_2}{2}\\&\leq \alpha \log\left \vert\Omega\right \vert
      \end{split}
      \end{equation}
               for some $0<\alpha<1/8$ if $0<\tau<1$ is chosen sufficiently small.

              Now the statement of the Theorem  \ref{thm_lower_bound_q} follows from \eqref{lowes_2} - \eqref{lowes_6}, \eqref{lowes_4} - \eqref{lowes_5},  \eqref{lowes_8} -  \eqref{lowes_9} and the Theorem 2.5 in \cite{tsybakov_book}.
  % % % % % % % % % % % % % % % % % % % % % % % % % % % % % % % % % % % % % % % % % % % % % %
 \section{Proof of Theorem \ref{thm_upper_bounds}.}  \label{proof_thm_upper_bounds}
  This proof essentially follows the scheme suggested in \cite{bunea_tsybakov_0} by adding an extension to the case of sub-Gaussian noise.
  Let $A\in \mathbb R ^{n_1\times n_2}$ be a fixed, but arbitrary matrix. Define for all $1\leq r\leq n_1n_2$
  \begin{equation*}
  \mathcal{B}_{r}=\left \{\bar A=A'-A\in \mathbb R ^{n_1\times n_2}\;:\; \Vert A' \Vert_{0}=r \right\}.
  \end{equation*}
  Let $\{J_k\}$, $k=1,\dots,\binom{n_1n_2}{r}$ be all the sets of matrix indices $(i,j)$ of cardinality $r$. Define
  \begin{equation*}
   \mathcal{B}_{r,k}=\left \{\bar A=(\bar a_{ij})\in \mathcal{B}_{r}\;:\;  a'_{ij}\not =0\;\iff\;(i,j)\in J_k \right\}
   \end{equation*}
   where $a'_{ij}=\bar a_{ij}+a_{ij}$. We have that $\mathrm{dim}(\mathcal{B}_{r,k})\leq r$. Let $\Pi_{r,k}(B)$ denote the projection of the matrix $B$ onto $\mathcal{B}_{r,k}$ and $\mathrm{pen}(A)=\lambda\Vert A\Vert_0\log\left (\dfrac{e\,n_1\,n_2}{\vert A\vert_0\vee 1}\right )$. By the definition of $\hat M$, for any $A\in \mathbb R ^{n_1\times n_2}$,
   \begin{equation*}
   \Vert Y-\hat M\Vert^{2}_{2}+\mathrm{pen}(\hat M)\leq \Vert Y-A\Vert^{2}_{2}+\mathrm{pen}(A).
   \end{equation*}
   Rewriting this inequality yields
   \begin{equation*}
   \begin{split}
    \Vert M-\hat M\Vert^{2}_{2}+\mathrm{pen}(\hat M)&\leq \Vert M-A\Vert^{2}_{2}+2\underset{(i,j)}{\Sigma}\xi_{ij}(\hat M-A)_{ij}+\mathrm{pen}(A)\\
    &\leq \Vert M-A\Vert^{2}_{2}+2\left (\underset{(i,j)}{\sum}\xi_{ij}\frac{(\hat M-A)_{ij}}{\Vert\hat M-A\Vert_{2}}\right )\Vert \hat M-A\Vert_{2}+\mathrm{pen}(A).
   \end{split}
     \end{equation*}
     For $B=(b_{ij})\in \mathbb R ^{n_1\times n_2} $ we set $V(B)=\underset{(i,j)}{\sum}\frac{\xi_{ij}\,b_{ij}}{\Vert B\Vert_{2}}$, then for any $a>1$
     \begin{equation} \label{hard_3}
       \begin{split}
        \left (1-\frac{1}{a}\right )\Vert M-\hat M\Vert^{2}_{2}+\mathrm{pen}(\hat M)&\leq \left (1+\frac{1}{a}\right )\Vert M-A\Vert^{2}_{2}+2aV^{2}(\hat M-A)+\mathrm{pen}(A).
       \end{split}
         \end{equation}
   Next, since $\mathbb R ^{n_1\times n_2}=\underset{r=0}{\overset{n_1n_2}{\bigcup}}\underset{k=1}{\overset{\binom{n_1n_2}{r}}{\bigcup}}\mathcal{B}_{r,k}$, we obtain
   \begin{equation*}
   2aV^{2}(\hat M-A)-\mathrm{pen}(\hat M)\leq \underset{0\leq r\leq n_1n_2}{\max}\;\underset{0\leq k\leq \binom{n_1n_2}{r}}{\max}\;\underset{\bar A\in \mathcal{B}_{r,k}}{\max}\left \{2aV^{2}(\bar A)-\mathrm{pen}(\bar A+A)\right \}.
   \end{equation*}
   Note that for $r=0$ we have that  $\mathcal{B}_{0}(A)=\{-A\}$ and $$2aV^{2}(-A)-\mathrm{pen}(-A+A)=2aV^{2}(A).$$
   Let $J_{\bar A}$ denotes the sparsity pattern of $\bar A=(\bar a_{ij})$, i.e.
   $$J_{\bar A}=\left \{(i,j)\in \mathbb N_{n_1\times n_2}\;:\;\bar a_{ij}\not = 0\right \},$$
   then for any $\bar A\in \mathcal{B}_{r,k}$
   $$V^{2}(\bar A)=\left (\underset{(i,j)\in J_{\bar A}}{\sum}\frac{\xi_{ij}\,\bar a_{ij}}{\Vert\bar A\Vert_{2}}\right )^{2}\leq \Vert \Pi_{r,k}(E)\Vert^{2}_{2}.$$
   This together with \eqref{hard_3} imply
    \begin{equation} \label{hard_1}
         \begin{split}
          \Vert M-\hat M\Vert^{2}_{2}&\leq \frac{a+1}{a-1}\Vert M-A\Vert^{2}_{2}+\frac{a}{a-1}\mathrm{pen}(A)+\frac{2a^{2}}{a-1}V^{2}(A)\\&+\frac{a}{a-1}\left [\underset{1\leq r\leq n_1n_2}{\max}\;\underset{0\leq k\leq \binom{n_1n_2}{r}}{\max}\left \{2a \Vert \Pi_{r,k}(E)\Vert^{2}_{2}-\lambda r\log\left (\dfrac{e\,n_1\,n_2}{r}\right )\right \}\right ].
         \end{split}
           \end{equation}
   By Assumption \ref{ass_noise}, the errors $\xi_{ij}$ are sub-gaussian. We will use the following tail bounds in order to control the last term in \eqref{hard_1}.
   \begin{lemma}\label{tail_bound}
   Let Assumption \ref{ass_noise} be satisfied. Then, there exists absolute constants $c_0, c_1, c_2, c_3>0$ such that for $K_1=K_0\,K^{2}$ with $K_0>0$ large enough

    \begin{equation}\label{tail_bound_proba}
    \mathbb P\left [\underset{1\leq r\leq n_1n_2}{\max}\;\underset{0\leq k\leq \binom{n_1n_2}{r}}{\max}\left \{\Vert \Pi_{r,k}(E)\Vert^{2}_{2}-K_1 r\log\left (\dfrac{e\,n_1\,n_2}{r}\right )\right \}\geq \Delta\right ]\leq c_1\exp\left \{-\frac{c_2\,\Delta^{2}}{K^{2}}\right \},  \end{equation}

    \begin{equation}\label{tail_bound_expectation}
       \hskip -1cm \mathbb E\left [\underset{1\leq r\leq n_1n_2}{\max}\;\underset{0\leq k\leq \binom{n_1n_2}{r}}{\max}\left \{\Vert \Pi_{r,k}(E)\Vert^{2}_{2}-K_1 r\log\left (\dfrac{e\,n_1\,n_2}{r}\right )\right \}\right ]\leq c_0\,K^{2} \end{equation}
  and
     \begin{equation}\label{tail_bound_V}
         \mathbb P\left [V^{2}(A)-K_1 \Vert A\Vert_{0}\geq \Delta\right ]\leq 2\exp\left \{-\frac{c_3\,\Delta^{2}}{K^{2}}\right \}  \end{equation}
  \end{lemma}
  Now \eqref{upper_expactation} follows from Lemma \ref{tail_bound} and \eqref{hard_1}.

  To prove \eqref{upper_proba}, note that by Lemma \ref{tail_bound} and \eqref{hard_1}, for $\lambda=2a\,K_0\,K^{2}$ there exist numerical constants $C,C_1,C_2>0$ such that
  \begin{equation*}
  \begin{split}
 & \mathbb P\left (\Vert M-\hat M\Vert^{2}_{2}\geq \underset{A\in \mathbb R^{n_1\times n_2}}{\inf}\left \{ \frac{a+1}{a-1}\Vert M-A\Vert^{2}_{2}+C\,\Vert A\Vert_{0}\,\log\left (\dfrac{e\,n_1\,n_2}{\Vert A\Vert_{0}}\right) \right \}+\dfrac{2a^{2}}{a-1}\Delta\right )\\&\leq
 \mathbb P\left (\left [\underset{1\leq r\leq n_1n_2}{\max}\;\underset{0\leq k\leq \binom{n_1n_2}{r}}{\max}\left \{ \Vert \Pi_{r,k}(E)\Vert^{2}_{2}-K_1 r\log\left (\dfrac{e\,n_1\,n_2}{r}\right )\right \}\right ]\geq \Delta/2\right )\\&\hskip 3 cm+\mathbb P\left (V^{2}(A)-K_1\Vert A\Vert_{0}\geq \Delta/2\right )\\& \leq C_1\exp\left \{-C_2\frac{\Delta}{K^{2}}\right \}
  \end{split}
  \end{equation*}
   which proves \eqref{upper_proba}.

   % % % % % % % % % % % % % % % % % % % % % % % % % % %
   % % % % % % % % % % % % % % % % % % % % % % % % % % % %
   % % % % % % % % % % % % % % % % % % % % % % % % % % % % % % %
   \section{Proof of Lemma \ref{tail_bound}}
    We  have that

    \begin{equation*}
   \begin{split}
    p_{\Delta}\overset{\mathrm{def}}{=}&\mathbb P\left [\underset{1\leq r\leq n_1n_2}{\max}\;\underset{0\leq k\leq \binom{n_1n_2}{r}}{\max}\left \{ \Vert \Pi_{r,k}(E)\Vert^{2}_{2}-K_1 r\log\left (\dfrac{e\,n_1\,n_2}{r}\right )\right \}\geq \Delta\right ]
    \\&\leq
    \sum_{r=1}^{n_1n_2}\sum_{k=1}^{\binom{n_1n_2}{r}}\mathbb P\left [\Vert \Pi_{r,k}(E)\Vert^{2}_{2} \geq \Delta+K_1 r\log\left (\dfrac{e\,n_1\,n_2}{r}\right )\right ]\\
    &\leq
     \sum_{r=1}^{n_1n_2}\binom{n_1n_2}{r}\mathbb P\left [ \mathbb Z_{r} \geq \Delta+K_1 r\log\left (\dfrac{e\,n_1\,n_2}{r}\right )-2rK^{2}\right ]
   \end{split}
   \end{equation*}
   where  $\mathbb Z_{r}=\sum^{r}_{i=1}\xi^{2}_i-\mathbb E(\xi^{2}_i)$ and $\xi_1,\dots,\xi_r$ are i.i.d. random variables satisfying Assumption \ref{ass_noise}. Note that $\xi^{2}_i$ are sub-exponential random variables with $\Vert\xi^{2}_i\Vert_{\psi_{1}}\leq 2\,K^{2}$. Applying Bernstein-type inequality (see, e.g., Proposition 5.16 in \cite {vershynin}) and using that $\binom{n_1n_2}{r}\leq \left (\dfrac{e\,n_1\,n_2}{r}\right )^{r}$ we get
    \begin{equation*}
   \begin{split}
    p_{\Delta}&\leq
     2\sum_{r=1}^{n_1n_2}\binom{n_1n_2}{r}\exp\left \{-C_2\left (K_0\,r\,\log\left (\dfrac{e\,n_1\,n_2}{r}\right )+\frac{\Delta}{2\,K^{2}}\right )\right \}\\&=
     2\exp\left \{-\frac{C_2\,\Delta}{K^{2}}\right \}\sum_{r=1}^{n_1n_2}\left (\dfrac{e\,n_1\,n_2}{r}\right )^{r}\exp\left \{-C_2\,K_0\,r\,\log\left (\dfrac{e\,n_1\,n_2}{r}\right )\right \}.
   \end{split}
   \end{equation*}
   Taking $K_0$ large enough we get
    \begin{equation*}
   \begin{split}
    p_{\Delta}&\leq
     2\exp\left \{-\frac{C_2\,\Delta}{K^{2}}\right \}\sum_{r=1}^{\infty}\exp\left \{-r\log 2\right \}\leq C_1\exp\left \{-\frac{C_2\,\Delta}{K^{2}}\right \}.
   \end{split}
   \end{equation*}
    This proves \eqref{tail_bound_proba} and  easily implies the bound on expectation value \eqref{tail_bound_expectation}.

    To proof \eqref{tail_bound_V}, we apply Bernstein-type inequality
     to $V^{2}(A)=\sum_{(i,j)\in J_A}(\xi_{ij})^{2}$:
   \begin{equation*}
   \begin{split}& \mathbb P\left [\sum_{(i,j)\in J_A}\xi_{ij}^{2}-\mathbb E\left (\xi_{ij}^{2}\right )\geq K_1\Vert A\Vert_{0}-2\Vert A\Vert_{0}\,K^{2}+\Delta\right ]\\
   &\hskip 1 cm\leq \exp\left \{-C_2\left (K_0\,\Vert A\Vert_{0}-\Vert A\Vert_{0}+\frac{\Delta}{2\,K^{2}}\right )\right \}\leq
   2\exp\left \{-\frac{C_2\,\Delta}{K^{2}}\right \}.
   \end{split}
   \end{equation*}
   % % % % % % % % % % % % % % % % % % % % % % % %
    % % % % % % % % % % % % % % % % % % % % % % % % % % % % % % % % % % % % % % % % % % % % % % % % % % % % % % % % % % % % % % % % % % % % % % % % % % % % % %

      \section{Proof of Corollary \ref{corollary_upper_bounds_q}.}  \label{proof_corollary_upper_bounds_q}
    We use Theorem \ref{thm_upper_bounds}. First, taking $A=0$  in \eqref{thm_upper_3a},  we get
    \begin{equation}\label{corollary_q_1}
    \begin{split}
    \Vert M-\hat M\Vert^{2}_{2}&\leq  \frac{a+1}{a-1}\Vert M\Vert^{2}_{2} +\dfrac{2a^{2}}{a-1}\Delta\\&\leq \frac{a+1}{a-1}n_1\,s^{2/q} +\dfrac{2a^{2}}{a-1}\Delta
    \end{split}
       \end{equation}
     with probability at least $1-2\exp\left \{-\frac{C_2\,\Delta}{K^{2}}\right \}$.

     Now, choosing $A=M$, we obtain that
      \begin{equation}\label{corollary_q_2}
      \begin{split}
      \Vert M-\hat M\Vert^{2}_{2}&\leq  C\,K^{2}\,\Vert M\Vert_{0}\,\log\left (\dfrac{e\,n_1\,n_2}{\Vert M\Vert_{0}\vee 1}\right) +\dfrac{2\,a^{2}}{a-1}\Delta\\&\leq  C\,K^{2}\,n_1\,n_2 +\dfrac{2\,a^{2}}{a-1}\Delta
      \end{split}
     \end{equation}
     with probability at least $1-2\exp\left \{-\frac{C_2\,\Delta}{K^{2}}\right \}$.

     Finally,  Theorem \ref{thm_upper_bounds} implies that for any $1\leq s'\leq n_2/2$, all $a>1$ and any $\Delta>0$
      \begin{equation}\label{corollary2_1}
       \Vert M-\hat M\Vert^{2}_{2}\leq \underset{A\in \mathcal{A}(2s')}{\inf} \frac{a+1}{a-1}\Vert M-A\Vert^{2}_{2}+C\,K^{2}\,n_1\,s'\,\log\left (1+\dfrac{n_2}{2\,s'}\right) +\dfrac{2a^{2}}{a-1}\Delta
      \end{equation}
      with probability at least $1-2\exp\left \{-\frac{C_2\,\Delta}{K^{2}}\right \}$.
     Now we use the following lemma. %proven in \cite{Tsybakov_St_Flour}:
       \begin{lemma}\label{lemma_approx_l_q}
       Let $1\leq s'\leq n_2/2$ and $0<q\leq 2$. For any $M\in \mathcal{A}(q,s)$, there exists $A\in \mathcal{A}(2s')$ such that
       \begin{equation}\label{approx_l_q}
       \Vert M-A\Vert^{2}_2\leq s^{2/q}\,(s')^{1-2/q}n_1.
       \end{equation}
%       where
%       \begin{equation}
%       C(q) = \left\{
%           \begin{array}{ll}
%                        \left (\dfrac{q}{1-q}\right )^{2},
%          & \hskip 0.5 cm\mbox{if} \quad  0<q<1\\  \\
%          \dfrac{q}{2-q}\;,
%             &\hskip 0.5 cm \mbox{if} \quad 1\leq q\leq 2.
%               \end{array} \right.
%       \end{equation}
       \end{lemma}
 For the proof of this lemma, see Lemma 7.2 in \cite{Tsybakov_Seoul}  (case $0<q\le 1$) and the proof of Lemma 7.4 in \cite{Tsybakov_Seoul}  (case $1<q\le 2$).

      Now, \eqref{corollary2_1} and Lemma \ref{lemma_approx_l_q} imply that for any $1\leq s'\leq n_2/2$
      \begin{equation} \label{corollary2_2}
          \Vert M-\hat M\Vert^{2}_{2}\leq
          C\left(K^{2}\,n_1\,s'\,\log\left (1+\dfrac{n_2}{s'}\right)+s^{2/q}\,(s')^{1-2/q}n_1 +\Delta\right).
         \end{equation}
         The terms depending on $s'$ on the right side of \eqref{corollary2_2} are balanced by choosing
         \begin{equation*}
         s'=\left \lfloor c'\dfrac{s}{K^{q}}\left (\log\left (1+n_2\,K^{q}s^{-1}\right )\right)^{-q/2} \right \rfloor
         \end{equation*}
         with suitable constant $c'>0$. With this choice of $s$ we get
         \begin{equation}\label{corrolary_q_3}
                \Vert M-\hat M\Vert^{2}_{2}\leq C\left(n_1\,s\,K^{2-q}\left  (\log\left (1+n_2\dfrac{K^{q}}{s}\right )\right)^{1-q/2} +\Delta\right).
               \end{equation}
     The inequalities \eqref{corollary_q_1}, \eqref{corollary_q_2} and \eqref{corrolary_q_3} imply the statement of the Corollary \ref{corollary_upper_bounds_q}.
\section*{Acknowledgments}

 The authors want to thank L.A. Bassalygo for suggesting a shorter proof of Lemma \ref{selection_lemma}. This work was supported  by the French National Research Agency (ANR) under the grants
ANR-13-BSH1-0004-02, ANR - 11-LABEX-0047, and by GENES.
It was also supported by the "Chaire Economie et Gestion des Nouvelles Donn\'ees", under the auspices of Institut Louis Bachelier, Havas-Media and Paris-Dauphine.

\end{document}